\documentclass{amsart}

\usepackage[T1]{fontenc}
\usepackage{times}
\usepackage{amssymb,epsfig,verbatim,xy}

\theoremstyle{plain}

\usepackage[T1]{fontenc}
\usepackage[utf8]{inputenc}
\usepackage[english]{babel}
\usepackage{amssymb}
\usepackage{amsthm}
\usepackage{graphics}
\usepackage{amsmath}
\usepackage{amstext}
\usepackage{multirow}
\usepackage{hyperref}

\usepackage{graphicx}
\usepackage{amsmath}
\usepackage{amssymb}
\usepackage{amsthm}
\usepackage{amstext}
\usepackage{epstopdf}
\usepackage{mathrsfs}
\usepackage{amscd}
\usepackage{color}

\newtheorem{thm}{Theorem}[section]
\newtheorem{dfn}[thm]{Definition}
\newtheorem{lemma}[thm]{Lemma}
\newtheorem{prop}[thm]{Proposition}

\newtheorem{THM}{Theorem}

\theoremstyle{remark}
\newtheorem{remark}[thm]{Remark}

\newcommand{\mb}{\mathbb}
\newcommand{\mc}{\mathcal}

\newcommand{\C}{\mb C}

\newcommand{\F}{\mc F}

\newcommand{\II}{\mathrm{II}}

\DeclareMathOperator{\rank}{rank}
\DeclareMathOperator{\Sym}{Sym}
\DeclareMathOperator{\Grad}{Grad}

\DeclareMathOperator{\ord}{ord}

\newcommand{\EGG}[2][k,m]{\mathrm{E^{GG}_{#1}}(#2)}
\newcommand{\ED}[2][k,m]{\mathrm{E_{#1}}(#2)}

\numberwithin{equation}{section}

\addtocounter{section}{0}             
\numberwithin{equation}{section}       

\title{Extactic divisors for webs and lines on projective surfaces}
\author{Maycol Falla Luza}
\author{Jorge Vit\'orio Pereira}
\date{\today}

\begin{document}
\begin{abstract}
Given a web (multi-foliation) and a linear system on a projective surface we construct
divisors cutting out the locus where some element of the linear system has
abnormal contact with the leaf of the web. We apply these ideas to reobtain a classical result
by Salmon on the number of lines on a projective surface. In a different vein, we
investigate the number of lines and of disjoint lines contained in a projective surface and
tangent to a contact distribution.
\end{abstract}

\maketitle

\setcounter{tocdepth}{1}

\tableofcontents

\section{Introduction}

\subsection{Extactic divisors}
A smooth point $x$ of a (germ of) plane curve $C \subset \mathbb P^2$ is an extactic point of order $n$ if there exists a plane curve of degree $n$
which intersects $C$ with multiplicity $h^0(\mathbb P^2, \mathcal O_{\mathbb P^2}(n))$ at $x$. For example inflection points are the  extactic points of order one, and the
extactic points of order two are the sextactic points. In the literature the extactic points of order $n$ are also called $n$-flexes.

In \cite{MR1860669} one can find a construction of divisors on $\mathbb P^2$
attached to a foliation $\F$  which intersect the leaves of $\F$ precisely at the extactic points.
These are called the extactic divisors of $\F$. The $n$-th extact divisor of $\F$ is defined whenever the general leaf of $\F$
is not contained in an algebraic curve of degree at most $n$. Of course, if a curve $C$ is contained in an algebraic curve
of degree $d$ then every point of $C$ is a extactic point of order $n\ge d$. In particular, the $n$-th extactic divisor
of a foliation $\F$ contains all $\F$-invariant algebraic curves of degree at most $n$.  The extactic divisors
proved to be useful in the study of the Liouvillian integrability of polynomial differential equations, see for instance \cite{MR2776495}
and  \cite{MR2276503}.

In Section  \ref{S:extactic}  we revisit the construction of
extactic divisors  and reformulate it
using the language of invariant jet differentials.
Moreover, we show how to extend this construction to webs on surfaces. It is worth noting that the inflection divisor for webs on $\mathbb P^2$ was treated before by the first
author in his Phd thesis.

\begin{THM}\label{THM:degree}
Let $\mathcal{W}$ be a $d$-web of degree $r$ on $\mathbb{P}^2$. If the number of  algebraic curves of degree at most $n$  invariant by $\mathcal W$ is finite
then the $n$-th extactic divisor of $\mathcal W$ has degree
$$
\frac{n}{8}\cdot \left[  (n+1)(n+2)(4d + (n+3)(r-d)) + (n+3)(n^2 + 3n-2)(d-1)(d+2r)  \right].
$$
\end{THM}

As in \cite{MR2536234}, we carry the construction of the extactic divisor on arbitrary surfaces and for arbitrary linear systems.
As an application we give in Section \ref{S:Salmon} a  proof of a classical result of Salmon
which provides a bound for the number of lines on
projective surfaces contained in $\mathbb P^3$.

\subsection{Involutive Lines on projective surfaces}
Questions in experimental physics lead one of us  to the study of involutive lines (projectivization
of Lagrangian planes)  contained in surfaces in $\mathbb P^3$,  cf. \cite{Polynomial}.
Vaguely motivated by these problems we study the numbers of involutive lines
in a  surface $S \subset \mathbb P^3$ in Section \ref{S:involutive}.

If we set $\ell_i(d)$ as the number of involutive lines a degree $d$ smooth surface in $\mathbb P^3$ can have, then
our bound takes the following form.

\begin{THM}\label{T:bound}
If $S \subset \mathbb P^3$ is a smooth surface of degree $d\ge 3$ in $\mathbb P^3$
then the number of involutive lines in $S$ is at most $3d^2 -4d$, i.e. $\ell_i(d) \le 3d^2-4d$. Moreover,
\[
\overline{ \ell_i} = \limsup_{d\to \infty} \frac{\ell_i(d)}{d^2} \in [1,3] \,
\]
\end{THM}

The study of collections of pairwise disjoint lines (skew lines)  on projective surfaces is much more recent.
Set $s\ell(d)$ as the number of skew-lines a smooth surface of degree $d$ can have.
Miyaoka proved that  $ s \ell(d) \le 2d(d-2)$ when  $d\ge 4$ \cite{MR2569571}.
There are quartics containing $16$ skew lines (Kummer surfaces) thus $s\ell(4)=16$.
Rams \cite{MR2085146} exhibited examples of smooth surfaces which imply  $s\ell(d) \ge d(d-2)+2$ for $d\ge 6$; and Boissi\`{e}re-Sarti \cite{MR2341513} improved
Rams' lower bound to $s\ell(d) \ge d(d-2)+4$ when  $d \ge 7$ and $\gcd(d,d-2)=1$.

In the higher dimensional case, Debarre \cite{Debarre} and Starr \cite{Starr}  (independently)
proved that smooth hypersurfaces in $\mathbb P^{2m+1}$ contain a finite number of linearly embedded $\mathbb P^m$'s, $m$-planes for short.
Starr (loc. cit.) observes that there is a naive upper bound that grows like $d^{(m+1)^2}$ and suggests that there should exist a polynomial
bound with leading term $ ({(3m+1)!}/{2} - 1 ) d^{m+1}$.

The problem of bounding the number of pairwise disjoint $m$-planes does not seem to be studied so far.
Concerning pairwise disjoint involutive $m$-planes, we prove the following bound.

\begin{THM}\label{T:D}
If $X \subset \mathbb P^{2m+1}$ is a smooth hypersurface of degree $d\ge 3$ in $\mathbb P^{2m+1}$
then the maximal number of pairwise disjoint involutive $m$-planes in $X$ is at most $(d-1)^{m+1} + 1$. Moreover, when $X$ is a surface in $\mathbb P^3$, i.e. $m=1$, the bound is sharp
\end{THM}

When $m>2$, we do not know if the bound above is sharp.

\section{Jet differentials}\label{S:jets}

\subsection{Jet spaces}
Let $(X,V)$ be a directed complex manifold, i.e. $X$ is a manifold and $V\subset TX$ is a subbundle of the tangent bundle of $X$.
The space of $k$-jets of germs of curves tangent to $V$, denoted by $J_k V$, is by definition the set of equivalence classes $j_k(f)$ of germs of
curves $f : (\C,0) \to X$ which are everywhere tangent to $V$ (i.e. $f'(t) \in V$ for every $t \in (\C,0)$) modulo the equivalence relation
$f \sim g$ if and only if all the derivatives of $f$ and $g$ at $0$ coincide up to order $k$. The space $J_k V$ is a vector bundle over $X$ of rank $k \rank V$.

Notice that $J_1 V$ is naturally isomorphic to (the total space of) $V$.

\subsection{Jet differentials}
Let us recall the definition of jet differentials of order $k$ and degree $m$ introduced in \cite{MR609557}.
These are sections of vector bundles $\EGG{V^*}$ over $X$ with fibers equal to the space of polynomials $Q(f',f'',\ldots, f^{(k)})$
over the fibers of $J^k V$ of weighted degree $m$ with respect to the $\mathbb C^*$-action
\[
\lambda \cdot  ( f', f'', \ldots, f^{(k)}) =  ( \lambda f', \lambda^2 f'', \ldots, \lambda^k f^{(k)}) .
\]

For $i$ ranging from $0$ to $[m/k]$, set $S_i$  as
the set of polynomials as above having degree with respect to $f^{(k)}$ at most $i$. Since  the degree of $Q$ with respect to  $f^{(k)}$ is at most $[m/k]$, we have  the following filtration
for  $\EGG{V^*}$:
\[
\EGG[k-1,m]{V^*} = S_0 \subset S_1 \subset \ldots \subset S_{[m/k]} = \EGG{V^*}\, .
\]
For  $1\le i\le [m/k]$,  the  quotient $S_i/S_{i-1}$ is isomorphic to $\EGG[k-1,m-ki]{V^*} \otimes \Sym^i V^*$.
Since
 $\EGG[1,m]{V^*}$ is nothing but $\Sym^m V^*$, we can proceed  inductively to obtain a filtration of $\EGG{V^*}$ satisfying
\[
\Grad^{\bullet} \EGG{V^*} = \bigoplus_{i_1 + 2 i_2 + \ldots + ki_k = m} \Sym^{i_1} V^* \otimes \cdots \otimes \Sym^{i_k} V^*.
\]

\subsection{Action of jet differentials on vector fields}\label{S:action0}

Let $\mathcal L$ be a line-bundle over $X$ and $\sigma \in H^0(X,\EGG{V^*} \otimes \mathcal L)$ be
a jet differential of order $k$ and degree $m$ with coefficients in $\mathcal L$. Given a holomorphic vector field
$v \in H^0(X,V) \subset H^0(X,TX)$ everywhere tangent to $V$ we can define the action of $\sigma$ on $v$ as follows. For any point $x \in X$ there exists
a unique germ $f_x : (\C,0) \to (X,x)$  such that $v(f_x(t)) = f_x'(t)$. We set $\sigma(v)$ as the section of $\mathcal L$
which at $x \in X$ is obtained by applying $\sigma$ to $j_k(f_x(t))$.

For any complex number $\lambda \in \mathbb C$ we have that $\sigma(\lambda v) = \lambda^m \sigma(v)$. But if $a \in H^0(X,\mathcal O_X)$ is
a non constant function ( of course we are assuming $X$ is not compact here)
 there is no obvious similar  relation between $\sigma(a v)$ and $a^m \sigma(v)$ when $k \ge 2$.

\subsection{Multiplication and differentiation of jet differentials}

Given two jet differentials $\sigma_i \in H^0(X, \EGG[k_i,m_i]{ V^*})$, we can multiply them to obtain an element of
$H^0(X, \EGG{V^*})$ where $k = \max\{k_1, k_2\}$ and $m = m_1 + m_2$ which sends a $k$-th jet $j_k(f)$ to $\sigma_1(j_{k_1}(f))\cdot \sigma_2(j_{k_2}(f))$. Thus we
have a $\mathcal O_X$-linear  (commutative) multiplication morphism
\[
\EGG[k_1,m_1]{V^*} \otimes_{\mathcal O_X} \EGG[k_2,m_2]{V^*} \longrightarrow \EGG{V^*} \, .
\]

There is also a linear differentiation morphism of $\mathbb C$-sheaves (cf. \cite{MR609557} )
\begin{align*}
D : \EGG{V^*} &\longrightarrow \EGG[k+1,m]{V^*} \\
\sigma &\longmapsto \left( j_{k+1}(f) \mapsto \frac{d}{dt} \sigma(j_k(f)) )\right) \, .
\end{align*}
In terms of the action of jet differentials on vector fields, we have that
\[
(D \sigma)(v) = v ( \sigma(v) ) \, .
\]
Notice that $D$ is not $\mathcal O_X$-linear but satisfies a Leibniz rule:
\[
 D ( \sigma_1 \cdot \sigma_2 ) = \sigma_1 \cdot D \sigma_2 + \sigma_2 \cdot D \sigma_1 \, .
\]

\subsection{Invariant jet differentials}

In \cite{MR1492539}, Demailly defined a subbundle $\ED{V^*}$ of the bundle of jet differentials $\EGG{V^*}$
whose sections  consist of jet differentials of order $k$ and degree $m$
which are invariant by reparametrizations  tangent to the identity. More explicitly, a jet differential $\sigma$ is an invariant jet differential if and only if at every point $x \in X$ it satisfies
\[
\sigma ( j_k(f ) ) = \sigma( j_k ( f \circ \varphi ) ) \,
\]
for any germ $f:(\mathbb C,0) \to (X,x)$ and any germ of diffeomorphism $\varphi : (\C,0) \to (\C,0)$ with $\varphi'(0)=1$.   The sections of $\ED{V^*}$ are called
invariant jet differentials of order $k$ and degree $m$.

\subsection{Action on foliations}
Since invariant jet differentials are jet differentials by definition, they act on vector fields as explained in \S\ref{S:action0}. Given a invariant jet differential with coefficients in a line bundle $\mathcal L$, say $\sigma \in H^0(X,\ED{V^*}\otimes \mathcal L)$, its invariance under reparametrizations
implies that for any $v \in H^0(X,V) \subset H^0(X,TX)$ and any holomorphic function $f \in H^0(X,\mathcal O_X)$ (constant or not) we have the identity
\[
\sigma ( f v ) = f^m \sigma (v) \in H^0(X,\mathcal L)\, .
\]
Therefore, $\sigma$ acts not only on vector fields but also on vector fields with coefficients on line bundles. If $v$ now belongs to $H^0(X,V\otimes \mathcal M)$
where $\mathcal M$ is an arbitrary line bundle then
\[
\sigma(v) \in H^0(X, \mathcal L \otimes \mathcal M^{\otimes m})\,.
\]

\subsection{Action on webs (when rank $V$ is equal to $2$)}
Suppose now that $V$ has rank two. If $\sigma \in H^0(X,\ED{V^*} \otimes \mathcal L)$ is an invariant jet differential then given a symmetric vector field
$ v \in H^0(X,\Sym^d V \otimes \mathcal M)$ with coefficients in a line bundle $\mathcal M$ we can define the action of $\sigma$ on $v$ as follows. Outside a closed analytic subset $\Delta$, the symmetric vector field $v$ can be locally written as the product of $d$ vector fields $v_1, \ldots, v_d$: $v= v_1 \cdots v_d$. The local decomposition is of course not unique since we may replace $v_i$ by $a_i v_i$ where $a_i$ are holomorphic functions satisfying $\prod a_i = 1$. Nevertheless, we can choose one such decomposition and set
\[
\sigma(v) = \prod \sigma(v_i) \, .
\]
Since $\sigma(a_iv_i) = a_i^m \sigma(v_i)$, it follows that a different  decomposition of $v$ leads to the same result.
But this expression does not make sense a priori at the analytic subset $\Delta$ where the decomposition of $v$ in a product of vector fields fails to exist. Nevertheless it is clear that the result can be extended meromorphically through $\Delta$. Therefore if  $\sigma$ is an element of $H^0(X,\ED{V^*} \otimes \mathcal L)$  and $v \in H^0(X,\Sym^d V \otimes \mathcal M)$ then $\sigma(v)$ is a meromorphic section of $\mathcal L^{\otimes d} \otimes \mathcal M^{\otimes m}$.

\section{Extactic divisors}\label{S:extactic}

\subsection{Extactic divisors for foliations}
Let $X$ be a projective manifold and let  $|W| \subset \mathbb P H^0(X,\mathcal N)$ be a linear system of dimension $k \ge 1$. For any germ $f : (\mathbb C,0) \to X$ define
\[
   \sigma_W (f) = \det\left(
   \begin{array}{cccc}
      f_0(t) & f_1(t) & \cdots & f_k(t) \\
       f_0' (t)  &  f_1'(t)   & \cdots &  f_k'(t) \\
      \vdots & & & \\
      f_0^{(k)}(t)  & f_1^{(k)}(t)   & \cdots &  f_k^{(k)}(t) \\
   \end{array}\right)
\]
where $f_i(t) = s_i(f(t))$ for functions $s_0, \ldots, s_k$ expressing a basis of $W$ in a trivialization of $\mathcal N$ at a neighborhood of $f(0)$. These local expressions patch together to an invariant differential of order $k$ and degree $ m = k(k+1)/2$ with coefficients in $\mathcal N^{\otimes k+1}$, i.e. $\sigma_W$ can be interpreted as an element of $H^0(X,\ED{T^*X} \otimes \mathcal N^{\otimes k+1} )$.

Alternatively we can interpret $\sigma_W$ as follows. On any projective space $\mathbb P^k$ we have a natural invariant jet differential
\[
  \sigma \in H^0(\mathbb P^k, \ED[k,k(k+1)/2]{\Omega^1_{\mathbb P^k}} \otimes K_{\mathbb P^k}^*)
\]
defined as follows. Let $\gamma : ( \mathbb C,0) \to \mathbb P^k$  be a germ and consider an arbitrary lifting $\hat{\gamma} : (\mathbb C,0) \to \mathbb C^{k+1}-\{0\} $ under the natural
projection $\pi: \mathbb C^{k+1} - \{0\} \rightarrow \mathbb{P}^k$.
The jet differential  $\sigma$ maps $\gamma$ to  $\pi_*\hat{\gamma}'(t) \wedge \pi_*\hat{\gamma}''(t) \wedge \ldots \wedge \pi_*\hat{\gamma}^{(k)}(t) \in \gamma^* \bigwedge^k T \mathbb P^k \simeq \gamma^* K_{\mathbb P^k}^*$.
If we now consider the rational map $\varphi : X \dashrightarrow \mathbb P W^*\simeq \mathbb P^k$ associated to $|W|$ then $\sigma_W$ is nothing but the pull-back of $\sigma$ under $\varphi$.

\begin{prop}
 	If $f : (\C,0) \to X$ is a non-constant germ such that $\sigma_W(f(t))$ vanishes identically then  the image of any representative of $f$ is contained in an element of the linear system $W$.
\end{prop}
\begin{proof}
    If $\sigma_W(f(t)) =0$ then the local functions $f_0(t), \ldots , f_k(t)$ have zero Wronskian and therefore are linearly dependent.
\end{proof}

Given a foliation by curves $\mathcal F$ defined by a vector field $v  \in  H^0(X,TX \otimes T^*\mathcal F)$ the zero divisor of $\sigma_W(v) \in H^0(X, \mathcal N^{\otimes k+1} \otimes (T^*\mathcal F)^{\otimes m} )$ (when different from $X$)  is called in \cite{MR1860669}  the extactic divisor of $\mathcal F$ with respect to the linear system $W$.

\begin{prop}
With the previous notation, if $\sigma_W(v) $ vanishes identically then every leaf of $\mathcal F$ is contained in an element of the linear system $|W|$ and there exists a non-constant rational function $h \in \mathbb C(X)$ constant along the leaves of $\mathcal F$.
\end{prop}
\begin{proof}
	See \cite[Theorem 3]{MR1860669}.
\end{proof}

\subsection{Extactic divisors for webs on surfaces}

\begin{lemma}\label{L:degdisc}
	Assume $X$ is a surface and let  $\mathcal M$ be a line bundle on $X$.
    If $v$ is a holomorphic  section of $\Sym^d TX\otimes \mathcal M$  then the discriminant of $v$ is a section $\Delta(v)$  of $K_X^{-d(d-1)} \otimes \mathcal M^{\otimes 2(d-1)}$, i.e.
	\[
	\Delta(v) \in H^0(X,\det(TX)^{d(d-1)}\otimes \mathcal M^{\otimes 2(d-1)})
	\]	
\end{lemma}
\begin{proof}
	See \cite[Section 1.3.4]{MR2536234}.
\end{proof}

\begin{lemma}\label{L:actionwebsW}
	Let $X$ be a projective manifold and $|W| \subset \mathbb P H^0(X,\mathcal N)$ be a linear system of dimension $k \ge 1$.
	If $\sigma_W$ is the associated jet differential  and $v \in H^0(X,\Sym^d V \otimes \mathcal M)$  is a symmetric vector field with $\Delta(v)\neq 0$ then  the  polar divisor of $\prod \sigma_W(v_i)$  satisfies
	\[
	\left(\prod \sigma_W(v_i)\right)_{\infty} \le  \frac{k (k-1)}{2} (\Delta(v))_0 \, .
	\]
\end{lemma}
\begin{proof}
	Let $H$ be an irreducible component of the discriminant of $v$. At a neighborhood $U \simeq \mathbb D^n$ of a sufficiently general point  $p \in H$ we can choose local holomorphic coordinates $x_1, \ldots, x_{n-1},y$ such that $\Delta(v) = y^\delta\times u$ where $\delta\ge 1$ is an integer and $u$ is an unity of $\mathcal O_{X,p}$.
	Hence $H = \{ y=0 \}$. Since $V$ has rank two we can decompose $v$ at a neighborhood of $p$ as $w \times v_1 \cdots v_{d- r}$ where $w$ is a local section of $\Sym^r V$ and $v_1, \ldots, v_{d-r}$ are local sections of $V$. We can further assume that $w$ is indecomposable, i.e. cannot be expressed as a product of $w_1 w_2$ with $w_1$ and $w_2$ sections  of strictly positive symmetric powers of $V$. In this case the order of the discriminant of $w$ along $H$ is at least $r-1$, i.e., $r-1 \le \delta$. Since $r\ge2$, we also have that $r/2 \le \delta$.
	
	Notice that the fundamental group of  $U - H \simeq \mathbb D^{n-1} \times \mathbb D^*$  is $\mathbb Z$.  We can thus choose generators $e_1, e_2$ of $V$ and write
	\[
	w = \prod_{i=1}^r \left( a_i(x,y^{1/r}) e_1  + b_i(x, y^{1/r}  ) e_2 \right) \, ,
	\]
	for suitable holomorphic functions $a_i, b_i$. Set $w_i = a_i(x,y^{1/r}) e_1 + b_i(x,y^{1/r}) e_2$.  We claim that
	\[
	 \ord_H ( \sigma_W(w_i) ) \ge ( 1 + 2 + \ldots + (k-1) )  \left( \frac{1}{r} - 1  \right) = \frac{k(k-1)}{2} \left(\frac{1}{r} -1 \right).
	\]
	Indeed for the two first rows of the matrix used to compute $\sigma(w_i)$ the order
	along $H$ is non-negative. The chain rules shows that on the third row, a monomial $y^{ 1/r -1 }$ will appear. By the product rule, the fourth row will be a linear  combination of  the third row multiplied by $y^{1/r - 2}$ and another expression involving the monomial $y^{ 2(1/r -1)}$. The multiple of third row will be disregard when taking determinants. The claim follows by induction.

	Therefore the order of $\prod \left(\sigma_W(w_i)\right)$ along $H$ is at least
	\begin{equation}\label{E:boundpoles}
	   r  \frac{k(k-1)}{2} \left(\frac{1}{r} -1 \right) = \frac{k(k-1)}{2} \left(1-r \right) \ge - \frac{k(k-1)}{2} \delta.
	\end{equation}
	The lemma follows.
\end{proof}

\begin{remark}
Let $\mathcal W_2$ be a $2$-web on $(\mathbb C^2,0)$ having reduced discriminant equal to  $C=\{ y=0\}$.
If $C$ is not invariant by $\mathcal W_2$ then in suitable coordinates $\mathcal W_2$ is defined by
$v_2 = \frac{\partial}{\partial y}^2 - y \frac{\partial}{\partial x}^2$, see \cite[Lemma 2.1]{MR3038728}.
If we consider the linear system $|W|$   of dimension $k$ locally generated by $1, y, x, y^2, xy, x^2, \ldots, x^{s-2}y^2, x^{s-1}y, x^s$ when $k+1 = 3s$, then the polar divisor of $\sigma_W (v_2)$ has order exactly $k(k-1)/2$ over $C$. In fact, we have the local decomposition $v_2=u\cdot v$, where $u= \frac{\partial}{\partial y} - \sqrt{y} \frac{\partial}{\partial x}$ and $v= \frac{\partial}{\partial y} + \sqrt{y} \frac{\partial}{\partial x}$, so  over a general solution of $v$, say $f(t)=(2t^3/3, t^2)$, we see that $f^*(v)=\frac{1}{t}\frac{\partial}{\partial t}$ and that $f^*W$ is generated by $1, t^2,t^3, \ldots, t^{3s-1},t^{3s}$.
A simple computation shows that  $\sigma_W(f)$ has a pole at $t=0$ of order exactly $k(k-1)/2$.
Since $f$ ramifies over $C$ we see that $v$ contributes to the order of $\sigma_W(v_2)$ along $C$ with $-k(k-1)/4$. Since the other factor of $v_2$  also contributes with $-k(k-1)/4$ we deduce
that $\sigma_W(v_2)$ has a pole of order exactly $k(k-1)/2$ along $C$.
For the cases $k+1=3s-1$ or $3s-2$ we consider the linear systems $1, y, x, y^2, xy, x^2, \ldots, x^{s-2}y^2, x^{s-1}y$ and $1, y, x, y^2, xy, x^2, \ldots, x^{s-2}y^2$ respectively.
\end{remark}

\begin{dfn}
	The extactic divisor of a web $\mathcal W = [v]\in \mathbb PH^0(X,\Sym^d TX \otimes \mathcal M)$ with respect to a linear system $|W|$ of dimension $k$ is the divisor $\mathcal E(\mathcal W, W)$ defined by the  vanishing of
	\[
      \Delta(v)^{ k (k-1)/2} 	\cdot  \sigma_W(v) \, .
	\]
\end{dfn}

In the case $X=\mathbb{P}^2$, for  $\mathcal W=[v]\in \mathbb PH^0(\mathbb{P}^2,\Sym^d T\mathbb{P}^2(r-d))$ a $d$-web of degree $r$, we define the $n$-extactic divisor of $\mathcal{W}$, denoted by $\mathcal{E}_s(\mathcal{W})$, as the extactic divisor of the web with respect to the linear system $|W|=\mathbb{P}H^0 (\mathbb{P}^2, \mathcal{O}_{\mathbb{P}^2}(n))$.

\begin{thm}[Theorem \ref{THM:degree} of the introduction] \label{P:degree of extactic}
Let $\mathcal{W}$ be a $d$-web of degree $r$ on $\mathbb{P}^2$, then $\mathcal{E}_n (\mathcal{W})$, when different from $\mathbb{P}^2$, is a curve of degree
$$
\frac{n}{8}\cdot \left[  (n+1)(n+2)(4d + (n+3)(r-d)) + (n+3)(n^2 + 3n-2)(d-1)(d+2r)  \right].
$$
\end{thm}
\begin{proof}
The discriminant of the web is given by a
section of $ H^0(\mathbb{P}^2, \mathcal{O}_{\mathbb{P}^2}((d-1)(d+2r)))$ according to Lemma \ref{L:degdisc}. On  the other hand, our linear system has dimension $k=n(n+3)/2$, in particular the associated jet differential has degree $m=n(n+1)(n+2)(n+3)/8$. Therefore, $\sigma_W (v)$ is a meromorphic section of $\mathcal{O}_{\mathbb{P}^2}(n)^{\otimes (k+1)d}\otimes \mathcal{O}_{\mathbb{P}^2}(r-d)^{\otimes m}$. Finally, the extactic divisor $\mathcal{E}_s (\mathcal{W})$ is cut out by a section of $\mathcal{O}_{\mathbb{P}^2}(n)^{\otimes (k+1)d}\otimes \mathcal{O}_{\mathbb{P}^2}(r-d)^{\otimes m} \otimes \mathcal{O}_{\mathbb{P}^2}((d-1)(d+2r))^{k(k-1)/2}$.
\end{proof}

\begin{prop}\label{P:invariant curves}
Let $C$ be an irreducible curve and let $k_C$ be the dimension of the restriction to $C$ of
a  linear system $|W| \subset \mathbb P H^0(X,\mathcal N)$ of dimension $k$.
If $C$ is invariant by a $d$-web $\mathcal W$ defined by $v \in H^0(X,\Sym^d TX \otimes \mathcal L)$
then
\[
   (k - k_C) C \le \mathcal E(\mathcal W, W) \, .
\]
\end{prop}
\begin{proof}
Suppose first that $C$ is not contained in  $\Delta(v)$. Note that $l:=k-k_C$ is the number of linearly independent elements of $|W|$ containing $C$. In a local coordinate system we can write $C=\{f=0\}$, $v= v_1 \cdots v_d $ and $s_0, \dots, s_k$ a basis of $W$. If $l >0$, we can choose this base such that $s_i = f. \hat{s}_i$ for $i= 0, \ldots, l-1$. By hypothesis, we can also assume that $v_1(f)=L. f$ for some holomorphic function $L$. Therefore,
\[
   \sigma_W (v_1) = \det\left(
   \begin{array}{cccccc}
      f.\hat{s}_0 & f.\hat{s}_1 & \cdots  &f.\hat{s}_{l-1 }& \cdots & s_k \\
      f. L_{1,0}  &  f.L_{1,1}   & \cdots  &f.L_{1,l-1} &\cdots &  v_1(s_k) \\
      \vdots & & & \\
      f.L_{k,0} & f.L_{k,1}&\cdots &f.L_{k,l-1}  & \cdots &  v_{1}^{(k)}(s_k) \\
   \end{array}\right)
\]
which has $f^{k-k_C}$ as a factor. The case where $C$ is contained in $\Delta(v)$  is analogous using the decomposition as in the proof of Lemma \ref{L:actionwebsW}.
\end{proof}

\begin{prop}\label{P:extW=0}
	Let $v \in H^0(X,\Sym^d TX \otimes \mathcal L)$ be a symmetric vector field with coefficients in $\mathcal L$ defining an indecomposable $d$-web $\mathcal W$.
	If $\sigma_W(v) $ vanishes identically then every leaf of $\mathcal W$ is contained in an element of the linear system $|W|$.
\end{prop}
\begin{proof}
Let $C$ be an invariant curve not contained in $\Delta(v)$, then in a neighborhood of a generic point we can decompose $v = v_1 \cdots v_d$. By assumption we have $\sigma_W (v_i) =0$  for some $i$ but the transitivity of the monodromy of the web implies that $\sigma_W (v_j) =0$ for $j=1. \ldots, d$ and therefore $C$ is an element of the linear system. The case where $C$ is contained in $\Delta(v)$ is analogous.
\end{proof}

\begin{remark}
Propositions   \ref{P:invariant curves} and \ref{P:extW=0}  provide a useful tool to bound the number of curves (in a given linear system) invariant
by a given web. This was one of the original motivations to introduce the extactic divisors for foliations, cf. \cite{MR1860669}. Another motivation
was to be able to explicitly determine the invariant curves of a given foliation. Unfortunately, in the case of webs our  approach to define the extactic
divisors does not provide explicit formulas for them.
\end{remark}

\section{Lines  on projective surfaces}\label{S:Salmon}

\subsection{Second fundamental form}
Let $X \subset \mathbb P^N$ be a submanifold. The second fundamental form of $X$ ( see \cite[Section 1.b]{MR559347} ) is a morphism of $\mathcal O_X$-modules
\[
\II: \Sym^2 TX \longrightarrow NX \, .
\]
If $v \in T_xX$ then $\II(v,v)$ is proportional to  the projection on the normal bundle of $X$ at $x$ of the osculating plane of any curve through $x$ with tangent space at $x$ genenerated by $v$. Dualizing the morphism $\II$ and tensoring the result by $NX$ we obtain  $\omega_{\II} \in H^0(X,\Sym^2 \Omega^1_X \otimes NX)$.

When $X\subset \mathbb P^3$ is a non-degenerate  surface (i.e. not contained in a plane) then the second fundamental form induces
a $2$-web $\mathcal W_{\II}$ on $X$  defined by $\omega_{\II}$.
Since $\Sym^2 \Omega^1_X \simeq K_X^{\otimes 2} \otimes \Sym^2 TX$
we obtain
\[
v_{\II} \in H^0(X,\Sym^2 TX \otimes K_X^{\otimes 2} \otimes NX)
\]
defining $\mathcal W_{\II}$ as well.

We collect in the next proposition a number of well-known properties of the second fundamental form of surfaces in $\mathbb P^3$ which will be useful in what follows.

\begin{prop}\label{P:basicII}
Let $X$ be an irreducible  surface contained in an open subset of $\mathbb P^3$. 	The following assertions hold true.
\begin{enumerate}
	\item The second fundamental form vanishes identically on $X$ if and only if $X$ is an open subset of a $\mathbb P^{2}$ linearly embedded in $\mathbb P^3$.
	\item The discriminant $\Delta(v_{\II})=\Delta(\omega_{\II})$ vanishes identically if and only $X$ is a cone or $X$ is  the tangential surface of a curve $C \subset \mathbb P^3$.
    \item If  $i:C \to X\cap \mathbb P^2$ is  the inclusion of a planar curve  satisfying $i^*\omega_{\II}=0$ then $C$ is a line or $C$ is contained in the discriminant of $\omega_{\II}$
\end{enumerate}	
\end{prop}
\begin{proof}
Item (1) is proved in   \cite[(1.51)]{MR559347}. For item (2) see (1.52) loc. cit. Let's proof item (3). Assume $C \subset X \cap \mathbb P^2$ is not a line
and satisfies $i^* \omega_{\II} =0$. Therefore for a general point of $C$, its osculating plane is tangent to $X$.  It follows that the intersection of $X$ and
and the $\mathbb P^2$ containing $C$ is non-reduced and consequently  $C$ is contained in the discriminant of $\omega_{\II}$.
\end{proof}

\subsection{Salmon's Theorem}

\begin{thm}
 Let $X \subset \mathbb P^3$ be an irreducible  surface of degree $d$.
 If $X$  is not a ruled surface then there exists  $s \in H^0(X, \mathcal O_X(11d -24))$
 cutting out all  lines contained in $X$. In particular the number of lines contained  in a smooth projective surface
 of degree $d\ge 3$  is at most $d(11d - 24)$.
\end{thm}
\begin{proof}
	Let $p \in \mathbb P^3$ be a general point and let $|W|$
	be the linear system of the restriction to $X$ of hyperplanes containing $p$.  Consider the invariant jet differential associated to $|W|$, $\sigma_W \in H^0(X,\ED[2,3]{T^*X}\otimes \mathcal O_X(3))$. The action of $\sigma_W$ in $v_{II}$ gives us a meromorphic section of $\mathcal O_X(3)^{\otimes 2}\otimes (K_X^{\otimes 2} \otimes NX)^{\otimes 3}$ then the divisor $\mathcal{E}(\mathcal W_{\II}, W)$ is given by a holomorphic section $s$ of
	\[
	\mathcal O_X(3)^{\otimes 2}\otimes (K_X^{\otimes 2} \otimes NX)^{\otimes 3} \otimes \underbrace{K_X^{\otimes 2} \otimes NX^{\otimes 2}}_{\Delta(v_{II})^{\frac{2(2-1)}{2}}}
	  \simeq \mathcal O_X( 13 d -26  ) \, .
	\]
	
	Basic properties of the second fundamental form implies that every line contained in $X$ is invariant by it.  Proposition  \ref{P:invariant curves} implies that $s$ vanishes on  every line.
	
	Since $X$ is not uniruled by assumption, Proposition \ref{P:basicII} item (2) implies that the discriminant of $\omega_{\II}$ is not identically zero. If $s$ vanishes identically then Proposition \ref{P:extW=0} implies that through a general point of $X$ there exists a planar curve invariant by $\mathcal W_{\II}$. But planar curves invariant by $\mathcal W_{\II}$ and not contained in the discriminant of $\omega_{\II}$ are lines.
	This contradicts our assumption on the non-uniruledness of $X$, proving that $s$ is not identically zero.
	
	Consider now the linear projection $\pi :\mathbb P^3 \dashrightarrow \mathbb P^2$ with center at $p$. Its restriction to $X$, still denoted by $\pi$, has ramification divisor $R$ cut out by  $r \in H^0(X,\mathcal O_X(d-1))$.
	We claim that $(s)_0 \ge 2 R$. Fix a general point $x$ of $R$. At a neighborhood of it write $v_{\II} = w_1 \cdot w_2$. The  orbits of $w_1, w_2$  will have contact of order at least three with the element of $|W|$ tangent to $X$ at $x$. This is sufficient to show that the rank at $x$  of the matrix defining $\sigma_{W}(w_1)$ and $\sigma_{W}(w_2)$  is at most two. The claim follows.
	
	To conclude the proof of Salmon's Theorem it suffices to divide $s$ by $r^2$ to obtain a holomorphic section of $\mathcal O_X( 11d -24)$ vanishing along all the lines contained in $X$.
\end{proof}

The proof above is not very different of Salmon's proof. The divisor defined by $s$ coincides with the flecnodal divisor studied by Salmon.
Indeed, according to  \cite[page 138]{MR0007286}, every plane containing one of the null directions of $\II$ at a  point $p \notin \Delta(\II)$,
except the tangent plane, intersects the surface at a a planar curve having an inflection at $p$.  The vanishing of $\sigma(v_{\II})$ at $p$ implies that one of the asymptotic curves through 	$p$
order of contact at $p$ with this planar curve at least $3$. It follows that $p$ is also an inflection for the asymptotic curve. For a modern version of
Salmon's argument see \cite[Section 8]{MR3291856}.

Darboux \cite[p. 372]{zbMATH02708650} shows that through a general point there are exactly $27$ conics (curves)
which have abnormal contact with a surface $X\subset \mathbb P^3$. Thus there exists a $27$-web which
is tangent to every conic contained in $X$. Control on the normal bundle of this web
(i.e. a formula linear in the degree of $X$) would give a bound on the number of conics on a surface.

\section{Involutive lines}\label{S:involutive}

\subsection{Bound on the number of involutive lines}
Consider the projective space $\mathbb P^3$ endowed with
a contact structure $\mathcal C$ induced  by a constant  symplectic form $\sigma$  on $\mathbb C^4$.
If $\sigma = \sum_{i,j=0}^3 \lambda_{ij} dx_i \wedge dx_j$ is a symplectic form on $\mathbb C^4$ then
the associated contact structure $\mathcal C= \mathcal C_{\sigma}$ is defined by $\omega \in H^0 (\mathbb P^3, \Omega^1_{\mathbb P^3}(2))$
which in homogeneous coordinates can be written as
\[
\omega = i_R \sigma \, ,
\]
where $R= \sum_{i=0}^3 x_i \frac{\partial}{\partial x_i}$ is the radial (or Euler's) vector field and $i_R$ stands for the
interior product with $R$. A similar construction endows a contact structure over the projective space $\mathbb P^{{2m+1}}$.

A reduced and irreducible curve $C \subset \mathbb P^3$ is called an involutive curve (with respect to a contact distribution $\mathcal C$)
if $i^* \omega \in H^0(C_{sm},\Omega^1_{C_{sm}}(2))$ vanishes identically. Here $C_{sm}$ denotes the smooth locus of $C$ and $i: C_{sm} \to \mathbb P^3$ is the inclusion.

As in the introduction let $\ell_i(d)$ be  the number of involutive lines a degree $d$ smooth surface in $\mathbb P^3$ can have.

\begin{thm}[Theorem \ref{T:bound} of the introduction]
If $X \subset \mathbb P^3$ is a smooth surface of degree $d\ge 3$ in $\mathbb P^3$
then the number of involutive lines in $X$ is at most $3d^2 -4d$, i.e. $\ell_i(d) \le 3d^2-4d$. Moreover,
\[
\overline{ \ell_i} = \limsup_{d\to \infty} \frac{\ell_i(d)}{d^2} \in [1,3] \,
\]
\end{thm}
\begin{proof}
The restriction of the contact form  to $X$ gives a non-zero section $\omega$ of $\Omega^1_X\otimes \mathcal O_X(2)$.
Every involutive line contained in $X$ is invariant by the corresponding foliation $\mathcal{F}$.

The tangency locus between $\mathcal{F}$ and $\mathcal{W}_{\Pi}$ is cut out by a section of
\[
K_X^{\otimes 2} \otimes \mathcal O_X(2)^{\otimes 2} \otimes NX =
\mathcal O_X(3d-4).
\]
 Since it must contain every involutive line inside $X$, the first part of the theorem follows. For the last part take a surface $X$ in $\mathbb P^3$ of degree $d \ge 3$  defined by a homogenous polynomial of the form $p(x_0,x_1) + q(x_2,x_3)$, then $X$ has at least $d^2$ involutive lines with respect to the contact form $i_R(dx_0\wedge dx_1 + dx_2 \wedge dx_3)$. In fact, since $p$ and $q$ are binary forms they can factored as a product of $d$ linear forms, say $p(x_0,x_1) = \prod_{i=1}^d p_i(x_0,x_1)$ and
$q(x_2,x_3) =  \prod_{i=1}^d q_i(x_2,x_3)$. The lines $\{ p_i(x_0,x_1) = q_j(x_2,x_3) = 0 \}$ are all involutive and contained in $X$.
\end{proof}

\begin{remark}
As before, let $|W|$ be the linear system of the restriction to $X$ of hyperplanes containing a general point $p \in \mathbb{P}^3$. Then the extatic divisor of $\mathcal{F}$ with respect to $|W|$ contains every involutive line and is given by a section of $\mathcal{O}_X(3) \otimes \mathcal{O}_{X}(d-2)^{\otimes 3}=\mathcal{O}_X(3d-3)$ which gives us a worst bound for the number of involutive lines.
\end{remark}

\subsection{Pairwise disjoint involutive lines}

\begin{thm}[Theorem \ref{T:D} of the introduction]
If $X \subset \mathbb P^{2m+1}$ is a smooth hypersurface of degree $d\ge 3$ in $\mathbb P^{2m+1}$
then the maximal number of pairwise disjoint involutive $m$-planes in $X$ is at most $(d-1)^{m+1} + 1$. Moreover, when $m=1$ and $d \ge 6$ the bound is sharp.
\end{thm}

 In order to prove the theorem we need the following lemma.

\begin{lemma}\label{L:isolated}
If $X\subset \mathbb P^{2m+1}$ is a smooth hypersurface  of degree at least $3$, then pull-back of the contact form
$\omega \in H^0 (\mathbb P^{2m+1}, \Omega^1_{\mathbb P^{2m+1}}(2))$ to $X$ vanishes exactly at a subscheme of dimension zero and length
$\frac{(d-1)^{2m+2} -1  }{d -2  }$ isolated singularities.
\end{lemma}
\begin{proof}
Let $i: X \to \mathbb P^{2m+1}$ be the inclusion and suppose $i^* \omega \in H^0(X, \Omega^1_X(2))$ has non isolated singularities.
If $F \in \mathbb C[x_0, \ldots, x_{2m+1}]$ is a homogeneous polynomial of degree $d$ defining $X$, then
$dF_{|X}$ can be interpreted as sections of $\Omega^1_{\mathbb P^{2m+1}}(d)_{|X}$.
If  $Z$ is a positive dimensional irreducible component of  the singular set of $i^* \omega$ then the restriction
of $\omega$  to $Z$ must be proportional to the restriction of $dF$ to $Z$. Since $\omega$ has no singularities on $\mathbb P^{2m+1}$ we can write
$dF_{|Z} = s \cdot \omega_{|Z}$ where $s \in H^0(Z,\mathcal O_Z(d-2))$. If $d>2$ then $s$ vanishes on hypersurface of $Z$ and the same holds true for
$dF$. It follows that $X$ is singular along the zero locus of $s$ contrary to our assumptions.

It remains to determine the length of the zero scheme of $i^* \omega$, which can be done by computation the top Chern class  of $\Omega^1_X(2)$.
If $h=c_1(\mathcal O_X(1))$ then the Chern polynomial of $\Omega^1_X(2)$ is given by
\[
\frac{c({\Omega^1_{\mathbb P^{2m+1}}}_{|Y}(2))}{c(\mathcal O_X(2-d))} = \frac{ ( 1+h)^{2m+2}}{(1+2h)(1 - (d-2)h)}\, ,
\]
and the top Chern class of $\Omega^1_X(2)$ is $d=h^{2m}$ times the coefficient of $h^{2m}$.
Therefore
\[
c_{2m}(\Omega^1_X(2)) = d \left( \sum_{i=0}^{2m} \sum_{j=0}^{2m-i} \binom{2m+2}{i}(-2)^{j} (d-2)^{2m-i-j}  \right) \, .
\]
One can verify by induction that this last  quantity is equal to  $\frac{(d-1)^{2m+2} -1  }{d -2}$ as wanted.
\end{proof}

\begin{proof}[Proof of Theorem \ref{T:D}]
Let $j: P \to X$ be a linear inclusion of an $m$-plane in $X$, and $i: X \to \mathbb P^{2m+1}$ the inclusion of $X$ in $\mathbb P^n$.
Consider the exact sequence
\[
0 \to N^* P(2) \to {\Omega^1_X}_{|P}(2) \to \Omega^1_P(2) \to 0  \, .
\]
Since $j^* i^* \omega =0$, it follows that $(i^* \omega)_{|P}$ is the image of a certain $\sigma \in H^0( P , N^* P(2))$.
The intersection of the singular set of $i^*\omega$ with $P$ coincides (set-theoretically) with the singular set of $\sigma$. Moreover,
as a simple local computation shows, the length of zero scheme of $\sigma$ is at least the length of the restriction of the zero scheme of $i^* \omega$
to any neighborhood of $P$. If $N$ is the number of pairwise disjoint $m$-planes in $X$ then we can write
\[
 N c_m(N^*P(2)) \le c_{2m}(\Omega^1_X(2)) \, .
\]
But $c(N^*P(2)) = c({\Omega^1_X}_{|P}(2)) \cdot c(\Omega^1_P(2))^{-1}$ from which we deduce
\[
c(N^*P(2)) = \frac{ (1+h)^{2m+2} } { (1+2h)(1-(d-2)h)} \cdot \frac {1+ 2h }{(1+h)^{m+1} } = \frac{ (1+h)^{m+1} } { (1-(d-2)h)}
\]
where $h = c_1(\mathcal O_P(1))$. The coefficient of $h^m$ is exactly $\frac{(d-1)^{m+1} -1  }{d -2}$. This computation together with
Lemma \ref{L:isolated} implies
\[
N \le (d-1)^{m+1} +1
\]
as claimed. This concludes the firts part of Theorem \ref{T:D}. The following example from \cite{MR2085146} will show that the bound is sharp when $m=1$. Let us consider the surface
$$
S_{d} = \{ x_{0}^{d-1}x_1  + x_{1}^{d-1}x_2 + x_{2}^{d-1}x_3 + x_{3}^{d-1}x_0 =0\}
$$
where $d \geq 6$. Then $S_d$ contains the $d(d-2)+2$ skew lines $(\alpha x: \beta y : x : y)$ for $(\alpha, \beta)$ satisfying $\alpha = -\beta^{d-1}$, $\beta^{(d-1)^{2}+1} = (-1)^{d}$. Moreover, these lines are involutive with respect to the contact form $i_R(dx_0\wedge dx_2 + dx_1 \wedge dx_3)$.

\end{proof}

\end{document}